\documentclass[12pt, a4paper]{amsart}
\usepackage{eurosym}
\usepackage{amssymb}
\usepackage{amsmath}
\usepackage{amssymb}
\usepackage{amsthm}
\usepackage{bbm}
\usepackage{bm}
\usepackage{mathtools}
\usepackage{mathrsfs}
\usepackage[T1]{fontenc}
\usepackage[usenames,dvipsnames,svgnames,table]{xcolor}
\usepackage{esint}
\usepackage{combelow}
\usepackage{enumitem}
\usepackage[colorlinks=true,citecolor=black,urlcolor=NavyBlue,
linkcolor=black]{hyperref}
\usepackage{graphicx}
\usepackage{tikz}
\usepackage{geometry}
\usepackage{booktabs}
\usepackage{tikz-cd}

\allowdisplaybreaks
\makeatletter
\newcommand*\bigcdot{\mathpalette\bigcdot@{.5}}
\makeatother
\newtheorem{theorem}{Theorem}
\newtheorem{lemma}[theorem]{Lemma}
\newtheorem{proposition}[theorem]{Proposition}

\newtheorem{corollary}[theorem]{Corollary}

\theoremstyle{remark}
\newtheorem{remark}[theorem]{Remark}

\theoremstyle{definition}
\newtheorem{definition}{Definition}
\normalsize

\renewcommand{\geq}{\geqslant}

\renewcommand{\leq}{\leqslant}

\begin{document}
\title[Comparison of systems of linear pde]{Comparison of\\  polynomial matrix differential operators}
\author{Eduard Curc\u{a} and Bogdan Rai\cb{t}\u{a}}

\begin{abstract}
We characterize matrix polynomials $P,Q$ such that the inequality  
\begin{equation*}
\left\Vert Q(D)u\right\Vert _{L^{2}}\leq C\left\Vert P(D)u\right\Vert
_{L^{2}},\quad\text{for all }u\in C_c^\infty(\Omega),  
\end{equation*}
holds on bounded open sets $\Omega$. We also characterize the operators $P,Q$
for which the linear continuous embedding above is compact, i.e., if $u_n\in
C_c^\infty(\Omega)$ are such that $(P(D)u_n)_{n\geq 1}$ is bounded in $L^2$, then $%
(Q(D)u_n)_{n\geq 1}$ is strongly compact in $L^2$. 
\end{abstract}

\maketitle

\section{Introduction}

Let $p:\mathbf{R}^{d}\rightarrow \mathbf{C}$ be a non-zero polynomial and $B$
be the open unit ball in $\mathbf{R}^{d}$. A remarkable result of H\"{o}rmander (see \cite{Hor}) states that 
\begin{equation}
\left\Vert u\right\Vert _{L^{2}}\leq C\left\Vert p(D)u\right\Vert
_{L^{2}},\quad \text{for all }u\in C_{c}^{\infty }(B),  \label{Hor0}
\end{equation}%
where $D=-i\nabla $. We stress that this holds for \textit{any} nontrivial polynomial $p$ regardless the size of its zero set.

This main goals of this paper are, first, to extend H\"{o}rmander's
inequality to the case of systems (allowing for polynomial matrices $p$)
and, second, to characterize matrix polynomials $p$ such that the continuous
linear embedding in \eqref{Hor0} is compact (meaning that if $u_{n}\in
C_{c}^{\infty }(B)$ is such that $\sup_{n}\left\Vert p(D)u_{n}\right\Vert
_{L^{2}}<\infty $, then $u_{n}\rightarrow u$ in $L^{2}$ along a subsequence).

We next explain why both tasks are challenging and markedly different from %
\eqref{Hor0}.

\subsection{$L^{2}$-estimates}

As will become transparent, it is crucial that in \eqref{Hor0} both $u$ and $%
p$ are scalar. If $u$ is allowed to be a vector field, the inequality
clearly fails if $p$ is chosen to be, for instance, the divergence operator
(taking $u$ to be the $curl$ of a compactly supported potential). {As a
corollary of our main result, we will also construct an operator $p:\mathbf{R%
}^{d}\rightarrow \mathcal{M}_{2}(\mathbf{C})$ such that $\ker
_{C_{c}^{\infty }}p(D)=\{0\}$ and the inequality \eqref{Hor0} still fails.}
These examples show that the situation is substantially more complicated for
systems.

Our first main result result relies on the following generalization of %
\eqref{Hor0} from \cite{Hor}: Let $p,q:\mathbf{R}^{d}\rightarrow \mathbf{C}$
be polynomials. The inequality 
\begin{equation}
\left\Vert q(D)u\right\Vert _{L^{2}}\leq C\left\Vert p(D)u\right\Vert
_{L^{2}},\quad \text{for all }u\in C_{c}^{\infty }(B),  \label{Hor0.5}
\end{equation}%
holds if and only if $p$ \textit{dominates} $q$, meaning that there exists a
constant $C$ such that $\tilde{q}(\xi )/\tilde{p}(\xi )\leq C$ for all $\xi
\in \mathbf{R}^{d}$. Here $\tilde{p}\colon \mathbf{R}^{d}\rightarrow \lbrack
0,\infty )$ is defined by $\tilde{p}(\xi )^{2}=\sum_{\alpha }|\partial^{\alpha
}p(\xi )|^{2}$.

We will extend this concept to polynomial matrices, to which end we will use the Moore--Penrose generalized inverse $A^+\in \mathcal{M}_{N\times M}(\mathbf C)$ of a matrix $A\in \mathcal M_{M\times N}(\mathbf C)$ (see for instance \cite{Moore} and \cite{Penrose}), referred-to here simply as \textit{pseudoinverse} and the fact that the (pointwise) pseudoinverse of a matrix polynomial is a rational function, which follows immediately from the formula given in Theorem 3 of \cite{Dec}. 
\begin{definition}\label{def:dom}
    Let $P : \mathbf R^d\to\mathcal  M_{M\times N}(\mathbf C)$,  $Q : \mathbf R^d\to \mathcal M_{L\times N}(\mathbf C)$ be matrix polynomials. We write $QP^+=(q_{lj}/\Delta)_{l,j}$, where $q_{lj}$, $\Delta$ are polynomials. We say that $P$ \textbf{dominates} $Q$ (and write $Q\prec P$) if
    \begin{enumerate}
        \item[(i)] $\Delta$ dominates $q_{lj}$ for all $l,j$ and
        \item [(ii)] $\ker P(\xi)\subset \ker Q(\xi)$ for almost every $\xi\in \mathbb R^d$.
    \end{enumerate}
\end{definition}

We will show in Remark \ref{rem.inv} that our definition is independent of the choice of polynomials $q_{lj}$ and $\Delta$. We thus have the answer to our first question:

\begin{theorem}
\label{thm:main_ineq}  Let $P : \mathbf{R}^d\to \mathcal{M}_{M\times N}(%
\mathbf{C})$, $Q : \mathbf{R}^d\to \mathcal{M}_{L\times N}(\mathbf{C})$ be
matrix polynomials.  The following are equivalent: 

\begin{enumerate}
\item There is a constant $C>0$ such that  
\begin{equation*}
\left\Vert Q(D)u\right\Vert _{L^{2}}\leq C\left\Vert P(D)u\right\Vert
_{L^{2}},\quad\text{for all }u\in C_c^\infty(\Omega)^N.
\end{equation*}

\item $P$ dominates $Q$.
\end{enumerate}
\end{theorem}

Our characterization makes it transparent that inequality \eqref{Hor0} fails
in the general vectorial case, take for example $Q=I_2$ and 
\begin{align*}
P(\xi)=\left( 
\begin{matrix}
\xi_1^2+\xi_2^2 & \xi_1 \\ 
\xi_2 & 0%
\end{matrix}
\right).
\end{align*}

Our proof of the estimate in Theorem \ref{thm:main_ineq} relies on the scalar case due to H\"ormander \cite{Hor} and a  duality argument which carefully uses the algebra relations in Definition \ref{def:dom}.
\subsection{Compactness}

Consider the following simple estimate 
\begin{equation}
\left\Vert u\right\Vert _{L^{2}}\leq C\left\Vert \Delta u\right\Vert
_{L^{2}},\quad \text{for all }u\in C_{c}^{\infty }(B).  \label{R-K}
\end{equation}

By the Rellich--Kondrachov theorem, \eqref{R-K} leads to a compact
embedding.  On the other hand, simple examples such as $p(D)=D_{1}$ show
that the compactness corresponding to the embedding \eqref{Hor0} can fail in
general.

Our second main goal is to characterize the situations when the inequality in Theorem~\ref{thm:main_ineq} leads to a compact embedding, as explained above. Our
characterization is based on another result of H\"{o}rmander (see Theorem \ref{th.cH} below), which leads us to introduce another new notion:
\begin{definition}\label{def:cpt_dom}
   (a) Let $p,q : \mathbf R^d\to\mathbf C$ be polynomials. We say that $p$ \textbf{compactly dominates} $q$ (and write $q\prec _{c}p$) if $\tilde q(\xi)/ \tilde p(\xi)\to 0$ as $|\xi|\to\infty$.

   (b) Let $P : \mathbf R^d\to \mathcal M_{M\times N}(\mathbf C)$,   $Q : \mathbf R^d\to \mathcal M_{L\times N}(\mathbf C)$ be matrix polynomials. We write $QP^+=(q_{lj}/\Delta)_{l,j}$, where $q_{lj}$, $\Delta$ are polynomials. We say that $P$ \textbf{compactly dominates} $Q$ (and write $Q\prec _{c}P$)if
    \begin{enumerate}
        \item[(i')] $\Delta$ compactly dominates $q_{lj}$ for all $l,j$ and
        \item [(ii)] $\ker P(\xi)\subset \ker Q(\xi)$ for almost every $\xi\in \mathbb R^d$.
    \end{enumerate}
\end{definition}
We will show in Remark \ref{rmk:consistency} that compact domination indeed implies domination, so that the terminology is consistent.

We can now state our main compactness result:
\begin{theorem}
\label{thm:main_compact}  Let $P : \mathbf{R}^d\to \mathcal{M}_{M\times N}(%
\mathbf{C})$, $Q : \mathbf{R}^d\to \mathcal{M}_{L\times N}(\mathbf{C})$ be
matrix polynomials.  The following are equivalent: 

\begin{enumerate}
\item For any sequence $(u_{n})_{n\geq 1}$ in $C_{c}^{\infty }(\Omega )^{N}$
such that $\sup_{n\geq 1}\left\Vert P(D)u_{n}\right\Vert _{L}<\infty $,
there exists a subsequence $(u_{n_{k}})_{k\geq 1}$ such that $Q(D)u_{n_{k}}$
converges strongly in $L^{2}$.

\item $P$ compactly dominates $Q$.
\end{enumerate}
\end{theorem}


\subsection{Application to lower semicontinuity of variational integrals}

The original motivation for the present work stems from the study of lower
semicontinuity properties of integral functionals generalizing \cite{Morrey}%
, 
\begin{equation*}
\int_{\Omega} F(P(D)u(x))d x, 
\end{equation*}
which are also related to the general $A$-free framework discussed in \cite%
{FM99,R_pot}, leading to variations of the notion of quasiconvexity of the
integrand $F$. Except in a few instances \cite{Muller99,DLSz}, the majority
of results in this direction pertain to so-called \textit{constant rank}
operators $P(D)$ and attempts at a general theory seem to be very
challenging, see the remarks in \cite[Section 7]{KrRa}.

Here we will make some progress to this question under the stringent
technical assumption that $P$ compactly dominates all lower order operators $%
P^{(\alpha)}$, defined as $P^{(\alpha)}(\xi)=\partial^{\alpha}P(\xi)$. With
this notation, our result is as follows:

\begin{theorem}
\label{thm:main_lsc} Let $P:\mathbf{R}^{d}\rightarrow 
\mathcal{M}_{M\times N}(\mathbf{C})$ be a matrix polynomial such that $P$
compactly dominates $P^{(\alpha )}$, for any multiindex $\alpha \neq 0$. Let 
$F:\mathbf{C}^{M}\rightarrow \mathbf{R}$, with $|F(z)|\lesssim 1+|z|^{2}$
satisfy 
\begin{equation}
\int_{(0,1)^{d}}(F(z+P(D)u(x))-F(z))dx\geq 0,\quad \text{for all }z\in 
\mathbf{C}^{N},\,u\in C_{c}^{\infty }((0,1)^{d})^{N}. 
 \label{qc}
\end{equation}%
Suppose $(u_{n})_{n\geq 1}$ is a sequence in $C_{c}^{\infty }(B)^{N}$
such that $P(D)u_{n}\rightarrow P(D)u$ weakly in $L^{2}(\Omega )^{M}$ for
some $u\in \mathscr D^{\prime }(\Omega)^N$. Then 
\begin{equation*}
\liminf_{n\rightarrow \infty }\int_{B}F(P(D)u_{n}(x))dx\geq
\int_{B}F(P(D)u(x))dx.
\end{equation*}
\end{theorem}

The boundary condition imposed in the lower semicontinuity statement about
can be replaced with typical less restrictive assumptions, provided that
concentration at the boundary is ruled out. For instance, one can consider
distributions $u_n\in \mathscr D^{\prime }(\Omega)^N$ such that $P(D)u_{n}\rightarrow
v$ weakly in $L^{2}(B)^M$, provided that the integrals are taken over
sets $\omega\Subset B$ or if $F$ is assumed bounded from below.

The \textit{quasiconvexity} condition \eqref{qc} is necessary and
sufficient for the lower semicontinuity claim. However, in the absence of
the technical condition that $P$ compactly dominates $P^{(\alpha)}$ for all $%
\alpha$, it is unclear if \eqref{qc} is sufficient for lower
semicontinuity \cite[Section 5]{HNR}. This question is particularly
challenging in the absence of the so-called constant rank condition of $P$ 
\cite{Murat}, in which case \eqref{qc} is an $A$-quasiconvexity condition 
\cite{FM99}. Simple examples show that this is not the case for general
linear differential operators $P$. 

\section*{Acknowledgements}

The first author was supported by the National Science Centre, Poland, CEUS programme, project no. 2020/02/Y ST1/00072.

\section{Proofs of estimates and compactness}

\subsection{Definitions and old results}
\label{sec.def}

First we introduce some notation. All functions and distributions are
complex-valued. The open unit ball in $\mathbf{R}^d$ is $B$. We denote the
Moore--Penrose generalized inverse of a matrix $A\in \mathcal{M}%
_{M\times N}(\mathbf{C})$ by $A^+\in \mathcal{M}_{N\times M}(\mathbf{C})$ 
(\cite{Moore}), and refer to it as the \textit{pseudoinverse of $A$}. The
crucial properties used here are that $AA^+$ and $A^+A$ are the orthogonal
projections on $\mathrm{im\,}A$ and $\mathrm{im\,}A^*$ respectively,
together with the algebraic identities 
\begin{equation*}
AA^+A=A\quad\text{and}\quad A^+AA^+=A^+. 
\end{equation*}
Here $A^*$ denotes the conjugate transpose of $A$.

The Fourier transform of a Schwartz function on $\mathbf{R}^d$ is given by
\begin{equation*}
\mathcal{F }f(\xi)=\hat f(\xi)=\int_{\mathbf{R}^d} e^{-i \xi\cdot
x}f(x) dx\quad\text{for }f\in\mathscr S(\mathbf{R}^d)^N. 
\end{equation*}
The definition of the Fourier transform extends by duality to the space $%
\mathscr S^{\prime }(\mathbf{R}^d)^N$ of tempered distributions. This
contains the space $\mathcal{E}^{\prime }(\Omega)^{N}$ of distributions on $\mathbf R^d$ which have 
compact support contained in $\bar\Omega$, where $\Omega\subset\mathbf{R}^d$
is an open set.

Given a matrix polynomial $P\colon\mathbf{R}^d\to\mathcal{M}_{M\times N}(%
\mathbf{C})$, the differential operator $P(D)$ is defined using $D=-i
\nabla$ (the real partial derivatives $\partial_1,\partial_2,\ldots,%
\partial_d$ are regarded as polynomial variables, as partial differentiation
is a commutative operation). In particular, if $P(\xi)=\sum_{|\alpha|\leq k}
\xi^\alpha P_\alpha$, where $P_\alpha\in \mathcal{M}_{M\times N}(\mathbf{C})$%
, then 
\begin{equation*}
P(D)u=\sum_{|\alpha|\leq k} P_\alpha D^\alpha u=\sum_{|\alpha|\leq k}
(-i)^\alpha P_\alpha \partial^\alpha u\quad\text{for }u\in C^\infty(\mathbf{R%
}^d)^N. 
\end{equation*}
With this definition we have that $\mathcal{F}(P(D)u)(\xi)=P(\xi)\mathcal{F }%
u(\xi)$ for $u\in\mathscr S(\mathbf{R}^d)^N$.

Following \cite[Chapter X]{Hor2}, we introduce the class $\mathcal{K}$,
consisting of \textit{tempered weights}, i.e., functions $k:\mathbf{R}%
^{d}\rightarrow (0,\infty )$ such that there exist constants $C,a>0$ such
that 
\begin{equation*}
k(\xi +\eta )\leq (1+C|\xi |^{a})k(\eta ),\quad \text{for all }\xi ,\eta \in 
\mathbf{R}^{d}.
\end{equation*}

Also, as in \cite[Def. 10.1.6]{Hor2}, for any $p\in \lbrack 1,\infty )$ and
any $k\in \mathcal{K}$ we define the space $B_{p,k}(\mathbf{R}^{d})$ of all
the tempered distributions $u$ on $\mathbf{R}^{d}$ for which  
\begin{equation*}
\left\Vert u\right\Vert _{B_{p,k}}:=\left( \int_{\mathbf{R}^{d}}|\widehat{u}%
(\xi )k(\xi )|^{p}d\xi \right) ^{1/p}<\infty .
\end{equation*}

According to Theorem 10.1.7 in \cite{Hor2} the space $B_{p,k}(\mathbf{R}^{d})
$ is a Banach space and $C_{c}^{\infty }(\mathbf{R}^{d})$ is dense in $%
B_{p,k}(\mathbf{R}^{d})$. Note that, by Parseval's theorem we have $B_{2,1}(%
\mathbf{R}^{d})=L^{2}(\mathbf{R}^{d})$ (and, if $p\neq 2$, then $B_{p,1}(%
\mathbf{R}^{d})$ and $L^{p}(\mathbf{R}^{d})$ are different spaces). In what
follows, since we are interested only in $L^{2}$-estimates, we mainly work
with the spaces $B_{2,k}(\mathbf{R}^{d})$. However, in the proofs of the
main results it is convenient to use general tempered weights $k$ rather
than just the constant ones.

We also recall some notation from \cite[Chapter II]{Hor}. Suppose $p: 
\mathbf{R}^{d}\rightarrow \mathbf{C}$ is a polynomial function ($p$ is a
scalar polynomial with complex coefficients). We associate to $p$ the
function $\widetilde{p}:\mathbf{R}^{d}\rightarrow \lbrack 0,\infty )$
defined by 
\begin{equation*}
\widetilde{p}(\xi ):=\left( \sum_{\alpha \in \mathbf{N}^{d}}\left\vert
p^{(\alpha )}(\xi )\right\vert ^{2}\right) ^{\frac{1}{2}},\quad\text{for all 
$\xi \in \mathbf{R}^{d}$,}
\end{equation*}
where $p^{(\alpha )}:=\partial ^{\alpha }p$ for any multi-index $\alpha \in 
\mathbf{N}^{d}$.

Given two polynomial functions $p,q:\mathbf{R}^{d}\rightarrow \mathbf{C}$,
we say that $p$ \textit{dominates} $q$ (or that $q$ is weaker\footnote{%
Actually, in \cite[Chapter II, Section 2.3]{Hor} (or \cite[Definition
10.3.4, p. 30]{Hor2}) it is said that $q$ is weaker than $p$ if (\ref{Hor1})
holds.} than $p$), if there exists a constant $C>0$ such that 
\begin{equation*}
\widetilde{q}(\xi )\leq C\widetilde{p}(\xi ),\quad\text{for all $\xi \in 
\mathbf{R}^{d}$.}
\end{equation*}
In this case we write $q\prec p$.

We recall the following result of H\"{o}rmander (see \cite[Theorem 2.2, p.
179]{Hor}):

\begin{theorem}
\label{Hor2.2}Suppose $p,q:\mathbf{R}^{d}\rightarrow \mathbf{C}$ are two
polynomial functions. There exists a constant $C>0$ such that 
\begin{equation}
\left\Vert q(D)u\right\Vert _{L^{2}}\leq C\left\Vert p(D)u\right\Vert
_{L^{2}},\quad\text{for all $u\in C_{c}^{\infty }(B)$, }  \label{Hor1}
\end{equation}
if and only if $p$ dominates $q$.
\end{theorem}

In order to formulate a version of Theorem \ref{Hor2.2} adapted to the case of
matrix valued polynomials (Theorem \ref{thm:main_ineq}), we  need a more general notion of domination. Fix some $N\in 
\mathbf{N}^{\ast }$ and let $P:\mathbf{R}^{d}\rightarrow\mathcal{M}_{M\times
N}(\mathbf{C)}$, $Q:\mathbf{R}^{d}\rightarrow\mathcal{M}_{L\times N}(\mathbf{%
C)}$ be two polynomial functions such that $P$ is not identically $0$. Note
that there exists a matrix valued rational function that equals $P^{+}(\xi )$
a.e. on $\mathbf{R}^{d}$. The fact that the (pointwise)
pseudoinverse of a matrix polynomial is a rational function  follows
immediately from the formula given in Theorem 3 of \cite{Dec}. We can write 
\begin{equation}
P^{+}(\xi )=\frac{1}{\Delta (\xi )}A(\xi ), \quad\text{a.e. on $\mathbf{R}%
^{d}$,}  \label{P+A}
\end{equation}%
where $A:\mathbf{R}^{d}\rightarrow \mathcal{M}_{N\times M}(\mathbf{C)}$ is a
matrix valued polynomial and $\Delta :\mathbf{R}^{d}\rightarrow \mathbf{C}$
is a nonzero scalar polynomial. Hence, we can write 
\begin{equation}
Q(\xi )P^{+}(\xi )=\frac{1}{\Delta (\xi )}Q(\xi )A(\xi )=\left( \frac{%
q_{lj}(\xi )}{\Delta (\xi )}\right) _{l,j},\quad\text{a.e. on $\mathbf{R}^{d}
$,}  \label{QP+}
\end{equation}
where $q_{lj}:\mathbf{R}^{d}\rightarrow \mathbf{C}$ are scalar polynomials
for $l=1,\ldots,N$, $j=1,\ldots,M$.

We now recall from Definition
\ref{def:dom} that $Q\prec P$ if 
\begin{enumerate}
\item[(i)] $\Delta$ dominates $q_{lj}$ for all $l,j$ and 

\item[(ii)] $\ker P(\xi)\subset \ker Q(\xi)$ for almost every $\xi\in 
\mathbb{R}^d$. 
 \end{enumerate}
 In this case we say that $P$ \textit{dominates} $Q$.

\begin{remark}
\label{rem.inv}Note that the condition (i) does not depend on the chosen
representation for the entries of the rational matrix in (\ref{QP+}). More
generally, if we write 
\begin{equation*}
\frac{1}{\Delta (\xi )}Q(\xi )A(\xi )=\left( \frac{q_{lj}^{1}(\xi )}{%
p_{lj}^{1}(\xi )}\right) _{l,j=1,...,N},
\end{equation*}%
for some scalar polynomials $p_{lj},q_{lj}:\mathbf{R}^{d}\rightarrow \mathbf{%
C}$, then (i) (for a fixed pair $\Delta $, $A$) is satisfied if and only if $%
p_{lj}\prec q_{lj}$, for all $l=1,\ldots,N$, $j=1,\ldots,M$. This follows
from Theorem 10.4.1 in \cite{Hor2} and the discussion on p. 33 in \cite{Hor2}.
\end{remark}

\begin{remark}
\label{rem.Ker}By the definition of the pseudoinverse we have $P(P^{+}P-I)=0$%
. This and condition (ii) implies $QP^{+}P=Q$ and $P^{\ast }(P^{\ast
})^{+}Q^{\ast }=Q^{\ast }$, a.e. on $\mathbf{R}^{d}$. Moreover, condition
(ii) is equivalent with the following implication: If $P(D)u=0$ for $u\in
C^\infty_c(\mathbf{R}^d)^N$, then $Q(D)u=0$.
\end{remark}


We also introduce the ``algebraic'' version of compact domination. Given two polynomial functions $p,q:\mathbf{R}^{d}\rightarrow 
\mathbf{C}$, we say that $p$ \textit{compactly dominates} $q$ if there
exists a constant $C>0$ such that 
\begin{equation*}
\widetilde{q}(\xi )/\widetilde{p}(\xi )\rightarrow 0,
\end{equation*}%
when $|\xi |\rightarrow \infty $. In this case we write $q\prec _{c}p$.

As a consequence of \cite[Theorem 10.1.10]{Hor2} due to H\"{o}rmander (see Theorem \ref{th.cH} below) we have the following compactness result concerning the scalar case:
\begin{theorem}
\label{Hor2.2c}Suppose $p,q:\mathbf{R}^{d}\rightarrow \mathbf{C}$ are two
polynomial functions. If $q\prec _{c}p$ then, for any sequence $%
(u_{n})_{n\geq 1}$ in $C_{c}^{\infty }(B)^{N}$ with $\left\Vert
p(D)u_{n}\right\Vert _{L^{2}}\leq 1$, the sequence $(q(D)u_{n})_{n\geq 1}$
is relatively compact in $L^{2}$. Conversely, suppose that for any sequence $%
(u_{n})_{n\geq 1}$ in $C_{c}^{\infty }(B)^{N}$ with $\left\Vert
p(D)u_{n}\right\Vert _{L^{2}}\leq 1$, the sequence $(q(D)u_{n})_{n\geq 1}$
is relatively compact in $L^{2}$. Then, we have $q\prec _{c}p$.
\end{theorem}

Adapted to the case of matrix valued polynomials our notion of compact domination is the following, cf. Definition \ref{def:cpt_dom}(b).
 If we write $QP^{+}=(q_{lj}/\Delta )_{l,j}$ for matrix polynomials $P,Q$, we say that $P$ \textit{compactly dominates} $Q$
 if 

\begin{enumerate}
\item[(i')] $\Delta$ compactly dominates $q_{lj}$ for all $l,j$ and 

\item[(ii)] $\ker P(\xi)\subset \ker Q(\xi)$ for almost every $\xi\in 
\mathbb{R}^d$. 
\end{enumerate}
In this case we also  write $Q\prec _{c}P$.
With this extension of the notion of compact domination, Theorem \ref{thm:main_compact} is a generalization of Theorem \ref{Hor2.2c} above.

\begin{remark}\label{rmk:consistency}
Compact domination implies domination since $\tilde p>0$
everywhere unless $p = 0$ and similarly $\tilde\Delta>0$ everywhere unless $%
P=0$. Of course, if $p$ or $P$ are null, then either domination is only
possible in the trivial case when $q$ or $Q=0$.
\end{remark}

\subsection{$L^{2}$-estimates}

The main goal of this section is to establish Theorem \ref{thm:main_ineq}.
We proceed with the proof of sufficiency of domination.

\begin{proposition}
\label{prop.1}Suppose $P : \mathbf{R}^d\to \mathcal{M}_{M\times N}(\mathbf{C}%
)$, $Q : \mathbf{R}^d\to \mathcal{M}_{L\times N}(\mathbf{C})$ are matrix
polynomials. If $P$ dominates $Q$ then 
\begin{equation}
\left\Vert Q(D)u\right\Vert _{L^{2}}\leq C\left\Vert P(D)u\right\Vert
_{L^{2}},  \label{prop1-0}
\end{equation}%
for all $u\in C_{c}^{\infty }(B)^{N}$.
\end{proposition}


\begin{proof}
Fix some $k\in \mathcal{K}$. If $p,q:\mathbf{R}^{d}\rightarrow \mathbf{C}$
are scalar polynomials with $q\prec p$, then, by Theorem 10.3.6 in \cite%
{Hor2} we have the estimate 
\begin{equation}
\left\Vert q(D)v\right\Vert _{B_{2,k}}\leq C\left\Vert p(D)v\right\Vert
_{B_{2,k}},  \label{prop1-1}
\end{equation}%
for any $v\in C_{c}^{\infty }(2B)$. By duality, for any $f\in (B_{2,k}(%
\mathbf{R}^{d}))^{\ast }=B_{2,1/k}(\mathbf{R}^{d})$, with $\mathrm{supp\,}%
f\subset B$ we can find some $u_{1}\in B_{2,1/k}(\mathbf{R}^{d})$ such that 
\begin{equation}
p^{\ast }(D)u_{1}=q^{\ast }(D)f\text{,}  \label{prop1-2}
\end{equation}%
in the sense of distributions on $2B$. Choosing some function $\varphi \in
C_{c}^{\infty }(2B)$ with $\varphi =1$ on $B$, we have 
\begin{equation*}
p^{\ast }(D)(u_{1}\varphi )=(p^{\ast }(D)u_{1})\varphi +[p^{\ast
}(D),\varphi ]u_{1},
\end{equation*}%
and moreover since $\mathrm{supp\,}f\subset B$, one can write using (\ref%
{prop1-2}), 
\begin{equation}
p^{\ast }(D)u=q^{\ast }(D)f+R,  \label{prop1-3}
\end{equation}%
where $u:=u_{1}\varphi $ and $R:=[p^{\ast }(D),\varphi ]u_{1}$. By Theorem
10.1.4 and Theorem 10.1.15 in \cite{Hor2} we have that $u\in B_{2,1/k}(%
\mathbf{R}^{d})$. The remainder $R$  is a tempered distribution and (from
the definition of $R$) we have that $\mathrm{supp\,}R\subset \overline{%
2B\backslash B}$.

Taking the Fourier transform in (\ref{prop1-3}) we obtain 
\begin{equation*}
\widehat{u}=\frac{q^{\ast }}{p^{\ast }}\widehat{f}+\frac{1}{p^{\ast }}%
\widehat{R}\in \widehat{B_{2,1/k}}=L_{1/k}^{2},
\end{equation*}%
where $L_{1/k}^{2}$is the usual weighted space of those measurable functions $g$ on $%
\mathbf{R}^{d}$ for which the norm 
\begin{equation*}
\left\Vert g\right\Vert _{L_{1/k}^{2}}:=\left( \int_{\mathbf{R}^{d}}|g(\xi
)/k(\xi )|^{2}d\xi \right) ^{1/2},
\end{equation*}%
is finite.

Hence, we have proved the following: \smallskip

\noindent \textbf{Claim.} \textit{If }$p,q:\mathbf{R}^{d}\rightarrow \mathbf C$%
\textit{\ are scalar polynomials with }$q\prec p$\textit{, and }$f\in
B_{2,1/k}(\mathbf{R}^{d})$\textit{\ with }$\mathrm{supp\,}f\subset B$%
\textit{, then there exists a tempered distribution }$R$\textit{\ with }$%
\mathrm{supp\,}R\subset \overline{2B\backslash B}$\textit{\ and such that } 
\begin{equation*}
\frac{q^{\ast }}{p^{\ast }}\widehat{f}+\frac{1}{p^{\ast }}\widehat{R}\in 
\widehat{B_{2,1/k}}.
\end{equation*}

\bigskip

Consider now a vector valued function $f=(f_{1},...,f_{N})$ with $f_{j}\in
B_{2,1/k}(\mathbf{R}^{d})$\textit{\ with }$\mathrm{supp\,}f_{j}\subset B$,
for any $j=1,\ldots ,M$. Let $\Delta $ be the as in (\ref{P+A}). Since $%
Q\prec P$, we have $q_{lj}^{\ast }\prec \Delta ^{\ast }$, for any $%
l=1,\ldots ,N$, $j=1,\ldots ,M$. By using the Claim we have 
\begin{equation*}
\frac{q_{lj}^{\ast }}{\Delta ^{\ast }}\widehat{f_{j}}+\frac{1}{\Delta ^{\ast
}}\widehat{R_{lj}}\in \widehat{B_{2,1/k}},
\end{equation*}%
where $R_{ij}$ are tempered distributions with $\mathrm{supp\,}R_{lj}\subset 
\overline{2B\backslash B}$.

We now define a function $u\in B_{2,1/k}(\mathbf{R}^{d})$ by the formula 
\begin{equation*}
\widehat{u}:=(P^{\ast })^{+}Q^{\ast }\widehat{f}+\frac{1}{\Delta ^{\ast }}%
\widehat{R},
\end{equation*}%
where 
\begin{equation*}
R:=\left( \sum_{j=1}^{N}R_{1j},....,\sum_{j=1}^{N}R_{Nj}\right) ^{t}
\end{equation*}%
is a tempered distribution with $\mathrm{supp\,}R\subset \overline{%
2B\backslash B}$. By this we get 
\begin{equation}
\Delta ^{\ast }P^{\ast }\widehat{u}:=\Delta ^{\ast }P^{\ast }(P^{\ast
})^{+}Q^{\ast }\widehat{f}+{P^{\ast }}\widehat{R}.  \label{prop1-MP}
\end{equation}

Using Remark \ref{rem.Ker} we can rewrite (\ref{prop1-MP}) as 
\begin{equation*}
\Delta ^{\ast }P^{\ast }\widehat{u}:=\Delta ^{\ast }Q^{\ast }\widehat{f}+{P^*%
}\widehat{R}.
\end{equation*}

Taking the inverse Fourier transform we have 
\begin{equation}
\Delta ^{\ast }(D)P^{\ast }(D)u:=\Delta ^{\ast }(D)Q^{\ast }(D)f+{P^*(D)}R,
\label{prop1-4}
\end{equation}
in the sense of distributions on $\mathbf{R}^{d}$.

Let $v\in C_{c}^{\infty }(B)$ be a function and consider some $f\in
B_{2,1/k}(\mathbf{R}^{d})$\textit{\ with }$\mathrm{supp\,}f\subset B$ such
that $\left\Vert f\right\Vert _{B_{2,1/k}}=1$ and 
\begin{equation}
\left\Vert \Delta (D)Q(D)v\right\Vert _{B_{2,k}}\leq 2\left\langle \Delta
(D)Q(D)v,f\right\rangle .  \label{prop1-5}
\end{equation}

Consider $u\in B_{2,1/k}(\mathbf{R}^{d})$ with $\left\Vert u\right\Vert
_{B_{2,1/k}}\lesssim 1$ and $R$ as in (\ref{prop1-4}). Since $\mathrm{supp\,}%
v$ and $\mathrm{supp\,}R$ are disjoint, one can write: 
\begin{eqnarray*}
\left\langle \Delta (D)Q(D)v,f\right\rangle  &=&\left\langle v,\Delta ^{\ast
}(D)Q^{\ast }(D)f\right\rangle =\left\langle v,\Delta ^{\ast }(D)Q^{\ast
}(D)f+{P^{\ast }(D)}R\right\rangle  \\
&=&\left\langle v,\Delta ^{\ast }(D)P^{\ast }(D)u\right\rangle =\left\langle
\Delta (D)P(D)v,u\right\rangle  \\
&\lesssim &\left\Vert \Delta (D)P(D)v\right\Vert _{B_{2,k}}.
\end{eqnarray*}

By this and (\ref{prop1-5}) we get 
\begin{equation}
\left\Vert \Delta (D)Q(D)v\right\Vert _{B_{2,k}}\lesssim _{k}\left\Vert
\Delta (D)P(D)v\right\Vert _{B_{2,k}},  \label{prop1-6}
\end{equation}%
for any $v\in C_{c}^{\infty }(B)$ and for any $k\in \mathcal{K}$. Using now
Theorem 10.3.2 in \cite{Hor2} for $k:=1/\widetilde{\Delta }$, we have the
estimates 
\begin{equation*}
\left\Vert \Delta (D)Q(D)v\right\Vert _{B_{2,k}}\sim \left\Vert
Q(D)v\right\Vert _{L^{2}}\text{ \ \ and \ \ }\left\Vert \Delta
(D)P(D)v\right\Vert _{B_{2,k}}\sim \left\Vert P(D)v\right\Vert _{L^{2}},
\end{equation*}%
for any $v\in C_{c}^{\infty }(B)$, which together with (\ref{prop1-6}) gives
(\ref{prop1-0}).
\end{proof}


We conclude the proof of Theorem \ref{thm:main_ineq} by establishing
necessity of domination.

\begin{proposition}
\label{prop.2}Suppose $P : \mathbf{R}^d\to \mathcal{M}_{M\times N}(\mathbf{C}%
)$, $Q : \mathbf{R}^d\to \mathcal{M}_{L\times N}(\mathbf{C})$ are matrix
polynomials. If 
\begin{equation}
\left\Vert Q(D)u\right\Vert _{L^{2}}\leq C\left\Vert P(D)u\right\Vert
_{L^{2}},\quad\text{for all $u\in C_{c}^{\infty }(B)^{N}$,}  \label{prop2-0}
\end{equation}
then $P$ dominates $Q$.
\end{proposition}


\begin{proof}
We first show that $\ker P(\xi )\subseteq \ker Q(\xi )$ for almost all $\xi
\in \mathbf{R}^{d}$. With the notation in (\ref{P+A}) consider the matrix
polynomial%
\begin{equation*}
A^{1}:=AP-\Delta I.
\end{equation*}

By (\ref{P+A}) we have that $A^{1}=\Delta (P^{+}P-I)$, a.e. on $\mathbf{R}%
^{d}$ and we have (see Remark \ref{rem.Ker}), 
\begin{equation*}
P(D)A^{1}(D)v=0,
\end{equation*}%
on $\mathbf{R}^{d}$, for any $v\in C_{c}^{\infty }(B)^{N}$. Hence, applying (%
\ref{prop2-0}) for $u=A^{1}(D)v\in C_{c}^{\infty }(B)^{N}$, we get $%
Q(D)A^{1}(D)v=0$, and 
\begin{equation*}
Q(\xi )A^{1}(\xi )\widehat{v}(\xi )=0,
\end{equation*}%
for any $\xi \in \mathbf{R}^{d}$ and any $v\in C_{c}^{\infty }(B)^{N}$. This
immediately implies that $Q(\xi )A^{1}(\xi )=0$, for any $\xi \in \mathbf{R}%
^{d}$ (we can choose $v$ of the form $v=w\varphi $, where $\varphi \in
C_{c}^{\infty }(B)$ is scalar and $w\in \mathbf{C}^{N}$ is an arbitrary
vector). Now, we have $\Delta Q(P^{+}P-I)=0$, a.e. on $\mathbf{R}^{d}$ and
since $\Delta $ is not trivial, $QP^{+}P=Q$, a.e. on $\mathbf{R}^{d}$. From
this we obtain the inclusion of kernels.

\bigskip

Now we show that $q_{lj}\prec \Delta $, for all $l=1,\ldots,N$, $j=1,\ldots,M
$. The arguments are similar to those used in the proof of Proposition \ref%
{prop.1}. By duality, we get from (\ref{prop2-0}) that, for any $f\in
L_{c}^{2}(B)$ there exists some $u_{1}\in L_{c}^{2}(\mathbf{R}^{d})$ such
that 
\begin{equation*}
P^{\ast }(D)u_{1}=Q^{\ast }(D)f,
\end{equation*}%
on $B$. Suppose now that $\mathrm{supp\,} f\subset (1/2)B$ and choose some
function $\varphi \in C_{c}^{\infty }(B)$ with $\varphi =1$ on $(1/2)B$, we
have (as in (\ref{prop1-3}))

\begin{equation}
P^{\ast }(D)u=Q^{\ast }(D)f+R,  \label{prop2-1}
\end{equation}%
on $\mathbf{R}^{d}$, where $u:=u_{1}\varphi \in L_{c}^{2}(B)$ and $%
R:=[P^{\ast }(D),\varphi ]u_{1}\in W^{-n_{0},2}(\mathbf{R}^{d})$ (for some $%
n_{0}\in \mathbf{N}^{\ast }$) with $\mathrm{supp\,} R\subset \overline{%
B\backslash (1/2)B}$. By taking the Fourier transform in (\ref{prop2-1}) we
get 
\begin{equation*}
P^{\ast }\widehat{u}=Q^{\ast }\widehat{f}+\widehat{R},
\end{equation*}%
and hence 
\begin{equation}
(P^{\ast })^{+}P^{\ast }\widehat{u}=(P^{\ast })^{+}Q^{\ast }\widehat{f}%
+(P^{\ast })^{+}\widehat{R},  \label{prop2-2}
\end{equation}%
a.e., on $\mathbf{R}^{d}$. Note that $(P^{\ast })^{+}P^{\ast }$ is a
Hermitian projection onto the image of $P$. Hence, $(P^{\ast })^{+}P^{\ast }$
can be diagonalized by an unitary matrix, and its diagonal form contains
only $1$ or $0$. It follows that any entry of $(P(\xi )^{\ast })^{+}P(\xi
)^{\ast }$ is bounded by a constant independent of $\xi $. Now we get that $%
(P^{\ast })^{+}P^{\ast }\widehat{u}\in L^{2}(\mathbf{R}^{d})$ and by (\ref%
{prop2-1}), 
\begin{equation*}
(P^{\ast })^{+}Q^{\ast }\widehat{f}+(P^{\ast })^{+}\widehat{R}\in L^{2}(%
\mathbf{R}^{d}).
\end{equation*}

Writing $P^{+}=\Delta ^{-1}A$, as in (\ref{P+A}) we have $(P^{\ast
})^{+}=(\Delta ^{\ast })^{-1}A^{\ast }$ and we can write 
\begin{equation}
(\Delta ^{\ast })^{-1}(A^{\ast }Q^{\ast })\widehat{f}+(\Delta ^{\ast })^{-1}%
\widehat{R^{\prime }}\in L^{2}(\mathbf{R}^{d}),  \label{prop2-3}
\end{equation}%
where $R^{\prime }:=A^{\ast }(D)R$ is a distribution supported on $\overline{%
B\backslash (1/2)B}$.

Write $A^{\ast }Q^{\ast }=(q_{lj}^{\ast })_{l,j}$ and fix some $j\in
\{1,...,L\}$. Choosing $f$ of the form $f=(0,...,0,f_{j},0,...,0)^{t}\in
L_{c}^{2}((1/2)B)$ we get from (\ref{prop2-3}) that 
\begin{equation*}
(\Delta ^{\ast })^{-1}q_{lj}^{\ast }\widehat{f_{j}}+(\Delta ^{\ast })^{-1}%
\widehat{R_{lj}^{\prime }}\in L^{2}(\mathbf{R}^{d}),
\end{equation*}%
for any $l\in \{1,...,M\}$, where $R_{lj}^{\prime }$ are distributions
supported on $\overline{B\backslash (1/2)B}$. If we define now $u_{lj}\in
L^{2}(\mathbf{R}^{d})$ by 
\begin{equation*}
\widehat{u}_{lj}=(\Delta ^{\ast })^{-1}q_{lj}^{\ast }\widehat{f_{j}}+(\Delta
^{\ast })^{-1}\widehat{R_{lj}^{\prime }},
\end{equation*}%
we get 
\begin{equation*}
\Delta ^{\ast }(D)u_{lj}=q_{lj}^{\ast }(D)f_{j}+R_{lj}^{\prime },
\end{equation*}%
in the sense of distributions on $\mathbf{R}^{d}$. As in the proof of
Proposition \ref{prop.1} (see (\ref{prop1-4})--(\ref{prop1-6})) we get 
\begin{equation*}
\left\Vert q_{lj}(D)v\right\Vert _{L^{2}}\lesssim \left\Vert \Delta
(D)v\right\Vert _{L^{2}},
\end{equation*}%
for any $v\in C_{c}^{\infty }(B)$. Since the polynomials $q_{lj}$ and $%
\Delta $ are scalar, by Theorem 10.3.6 in \cite{Hor2}, we get $q_{lj}\prec
\Delta $.
\end{proof}





\subsection{Compactness}

The main goal of this section is to establish Theorem \ref{thm:main_compact}%
, to which end we proceed with a compactness criterion due to H\"ormander:


\begin{theorem}
\label{th.cH}(Theorem 10.1.10 in \cite{Hor2} formulated only for the spaces $%
B_{2,k}$) Suppose $\Omega \subset \mathbf{R}^{d}$ is a compact set with
nonempty interior. If $k_{1},k_{2}\in \mathcal{K}$ such that 
\begin{equation}
\frac{k_{2}(\xi )}{k_{1}(\xi )}\rightarrow 0\text{, \ when }|\xi
|\rightarrow \infty ,  \label{cH-0}
\end{equation}%
then, $B_{2,k_{1}}(\mathbf{R}^{d})\cap \mathcal{E}^{\prime }(\Omega )$
embeds compactly in $B_{2,k_{2}}(\mathbf{R}^{d})$. Conversely, if $%
B_{2,k_{1}}(\mathbf{R}^{d})\cap \mathcal{E}^{\prime }(\Omega )$ embeds
compactly in $B_{2,k_{2}}(\mathbf{R}^{d})$, then (\ref{cH-0}) holds.
\end{theorem}

Next, we prove the sufficiency of compact domination for $L^2$-compactness:

\begin{proposition}
\label{prop.3}Suppose $P,Q:\mathbf{R}^{d}\rightarrow \mathcal{M}_{N}(\mathbf{%
C})$ are matrix polynomials. If $Q\prec _{c}P$ then, for any sequence $%
(u_{n})_{n\geq 1}$ in $C_{c}^{\infty }(B)^{N}$ with $\left\Vert
P(D)u_{n}\right\Vert _{L^{2}}\leq 1$, the sequence $(Q(D)u_{n})_{n\geq 1}$
is relatively compact in $L^{2}$.
\end{proposition}


\begin{proof}
Consider some nontrivial $q_{lj}$ (otherwise the argument is trivial). We
have $\widetilde{q}_{lj}/\widetilde{\Delta }\in \mathcal{K}$ and that $%
\widetilde{q}_{lj}/\widetilde{\Delta }\rightarrow 0$ at $\infty $. By
Theorem \ref{th.cH} \ we get that, for any sequence $(\varphi _{n})_{n\geq 1}
$ in $L^{2}(\mathbf{R}^{d})$ with $\varphi _{n}\in \mathcal{E}^{\prime }(B)$
and $\left\Vert \varphi _{n}\right\Vert _{L^{2}}\leq 1$, the sequence $((%
\widetilde{q}_{lj}/\widetilde{\Delta })(D)\varphi _{n})_{n\geq 1}$ is
relatively compact in $L^{2}(\mathbf{R}^{d})$. Hence, there exists some
function $g\in L^{2}$ such that 
\begin{equation}
(\widetilde{q}_{lj}/\widetilde{\Delta })(D)\varphi _{n}\rightarrow g,
\label{prop.3-1}
\end{equation}%
in the $L^{2}$ norm, up to a subsequence. If we introduce $\widetilde{g}:=(%
\widetilde{\Delta }/\widetilde{q}_{lj})(D)g\in B_{2,k}$ (where $k:=%
\widetilde{q}_{lj}/\widetilde{\Delta }$) we have $\varphi _{n}\rightarrow 
\widetilde{g}$, in the sense of tempered distributions, up to a subsequence.
In particular, since $\varphi _{n}\in \mathcal{E}^{\prime }(B)$ for all $%
n\geq 1$, we get $\widetilde{g}\in \mathcal{E}^{\prime }(B)$. Now, since $%
\varphi _{n}-\widetilde{g}\in \mathcal{E}^{\prime }(B)$, by applying Theorem
10.3.2 in \cite{Hor2} and (\ref{prop.3-1}) we have 
\begin{equation*}
\left\Vert q_{lj}(D)(\varphi _{n}-\widetilde{g})\right\Vert _{B_{2,1/%
\widetilde{\Delta }}}\sim \left\Vert \widetilde{q}_{lj}(D)(\varphi _{n}-%
\widetilde{g})\right\Vert _{B_{2,1/\widetilde{\Delta }}}=\left\Vert (%
\widetilde{q}_{lj}/\widetilde{\Delta })(D)(\varphi _{n}-\widetilde{g}%
)\right\Vert _{L^{2}}\rightarrow 0,
\end{equation*}%
up to a subsequence. Hence, we have obtained that for any sequence $(\varphi
_{n})_{n\geq 1}$ in $L^{2}(\mathbf{R}^{d})$ with $\varphi _{n}\in \mathcal{E}%
^{\prime }(B)$ and $\left\Vert \varphi _{n}\right\Vert _{L^{2}}\leq 1$, the
sequence $((q_{lj}/\widetilde{\Delta })(D)\varphi _{n})_{n\geq 1}$ is
relatively compact in $L^{2}(\mathbf{R}^{d})$. By considering all these
facts in $i$ and $j$ we get that for any sequence $(f_{n})_{n\geq 1}$ in $%
L^{2}(\mathbf{R}^{d})^{N}$ with $f_{n}\in \mathcal{E}^{\prime }(B)^{N}$ and $%
\left\Vert f_{n}\right\Vert _{L^{2}}\leq 1$, the sequence $((QA/\widetilde{%
\Delta })(D)f_{n})_{n\geq 1}$ is relatively compact in $L^{2}(\mathbf{R}^{d})
$. We apply this fact to $f_{n}$ of the form $f_{n}:=P(D)u_{n}$, where $%
u_{n}\in \mathcal{E}^{\prime }(B)^{N}$. Hence, for any sequence $%
(u_{n})_{n\geq 1}$ in $\mathcal{E}^{\prime }(B)^{N}$ with $\left\Vert
P(D)u_{n}\right\Vert _{L^{2}}\leq 1$, the sequence $((QA/\widetilde{\Delta }%
)(D)P(D)u_{n})_{n\geq 1}$ is relatively compact in $L^{2}(\mathbf{R}^{d})^{L}
$. Recall that $\ker P\subset \ker Q$, a.e. on $\mathbf{R}^{d}$, which gives
us $QP^{+}P=Q$, a.e. on $\mathbf{R}^{d}$ (see Remark \ref{rem.Ker}).
Consequently (using also the identity $P^{+}=A/\Delta $), 
\begin{equation*}
\frac{QA}{\widetilde{\Delta }}P=\frac{\Delta }{\widetilde{\Delta }}QP^{+}P=%
\frac{\Delta }{\widetilde{\Delta }}Q,
\end{equation*}%
and $(\Delta /\widetilde{\Delta })(D)Q(D)u_{n})_{n\geq 1}$ is relatively
compact in $L^{2}(\mathbf{R}^{d})$. This implies that there exists some $%
G\in L^{2}(\mathbf{R}^{d})^{N}$ such that 
\begin{equation}
(\Delta /\widetilde{\Delta })(D)Q(D)u_{n}\rightarrow G,  \label{prop.3-2}
\end{equation}%
in the $L^{2}$ norm, up to a subsequence. Since $Q\prec P$ (the condition $%
Q\prec _{c}P$ is stronger) we have 
\begin{equation*}
\left\Vert Q(D)u_{n}\right\Vert _{L^{2}}\lesssim \left\Vert
P(D)u_{n}\right\Vert _{L^{2}}\leq 1,
\end{equation*}%
and the sequential Banach-Alaoglu theorem show that there exists $\widetilde{%
G}\in L_{c}^{2}(B)^{N}$ such that 
\begin{equation*}
Q(D)u_{n}\rightarrow \widetilde{G},
\end{equation*}%
weakly in $L^{2}$, up to a subsequence. By this (\ref{prop.3-2}) we have $%
(\Delta /\widetilde{\Delta })(D)\widetilde{G}=G$. Hence, by Theorem 10.3.2
in \cite{Hor2} (note that $Q(D)u_{n}-\widetilde{G}\in \mathcal{E}^{\prime
}(B)^{N}$) and (\ref{prop.3-2}) we get 
\begin{eqnarray*}
\left\Vert Q(D)u_{n}-\widetilde{G}\right\Vert _{L^{2}} &=&\left\Vert 
\widetilde{\Delta }(D)(Q(D)u_{n}-\widetilde{G})\right\Vert _{B_{2,1/%
\widetilde{\Delta }}}\lesssim \left\Vert \widetilde{\Delta }(D)(Q(D)u_{n}-%
\widetilde{G})\right\Vert _{B_{2,1/\widetilde{\Delta }}} \\
&\lesssim &\left\Vert (\Delta /\widetilde{\Delta })(D)Q(D)u_{n}-G\right\Vert
_{L^{2}}\rightarrow 0,
\end{eqnarray*}%
on a subsequence. This proves Proposition \ref{prop.3}.
\end{proof}



Before considering the converse of Proposition \ref{prop.3}, we give a
direct consequence that will be used in the proof of Theorem \ref{thm:main_lsc}.

\begin{corollary}
\label{cor.conv0}Suppose $P : \mathbf{R}^d\to \mathcal{M}_{M\times N}(%
\mathbf{C})$, $Q : \mathbf{R}^d\to \mathcal{M}_{L\times N}(\mathbf{C})$ are
matrix polynomials such that $Q\prec _{c}P$. Consider a sequence $%
(u_{n})_{n\geq 1}$ in $C_{c}^{\infty }(B)^{N}$ such that $%
P(D)u_{n}\rightarrow 0$ weakly in $L^{2}$. Then, $Q(D)u_{n}\rightarrow 0$
strongly in $L^{2}$, up to a subsequence.
\end{corollary}


\begin{proof}
Since the sequence $(P(D)u_{n})_{n\geq 1}$ is weakly convergent in $L^{2}$,
it is also bounded in $L^{2}$. Hence, by Proposition \ref{prop.3}, there
exists some $g\in L_{c}^{2}(B)^{L}$ such that $Q(D)u_{n}\rightarrow g$
strongly in $L^{2}$, up to a subsequence. In particular, we have 
\begin{equation}
\Delta (D)Q(D)u_{n}\rightarrow \Delta (D)g,  \label{conv0-1}
\end{equation}%
in the sense of distributions. The condition $Q\prec _{c}P$ implies that $%
QP^{+}P=Q$ a.e. (see Remark \ref{rem.Ker}) and we can rewrite $\Delta
(D)Q(D)u_{n}$ in (\ref{conv0-1}) as 
\begin{equation}
\Delta (D)Q(D)u_{n}=(\Delta (D)Q(D)P^{+}(D))P(D)u_{n}=(Q(D)A(D))P(D)u_{n},
\label{conv0-2}
\end{equation}%
in $L^{2}$. Since $P(D)u_{n}\rightarrow 0$ in the sense of distributions,
and $Q(D)A(D)$ is a differential polynomial, we have by (\ref{conv0-2}) that 
\begin{equation*}
\Delta (D)Q(D)u_{n}\rightarrow 0,
\end{equation*}%
in the sense of distributions. This combined with (\ref{conv0-1}) gives us. 
\begin{equation}
\Delta (D)g=0,  \label{Dg=0}
\end{equation}%
in the sense of distributions. Taking the Fourier transform in (\ref{Dg=0})
we get $\Delta \widehat{g}=0$ in the sense of tempered distributions. Since, 
$\widehat{g}$ is an $L^{2}$ function and $\Delta $ is nontrivial (the set $%
\{\xi \in \mathbf{R}^{d}$ $|$ $\Delta (\xi )=0\}$ is of Lebesgue measure $0$%
), we get $g\equiv 0$.
\end{proof}


\begin{proposition}
\label{prop.4}Suppose $P:\mathbf{R}^{d}\rightarrow \mathcal{M}_{M\times N}(%
\mathbf{C})$, $Q:\mathbf{R}^{d}\rightarrow \mathcal{M}_{L\times N}(\mathbf{C}%
)$ are matrix polynomials. Suppose that for any sequence $(u_{n})_{n\geq 1}$
in $C_{c}^{\infty }(B)^{N}$ with $\left\Vert P(D)u_{n}\right\Vert
_{L^{2}}\leq 1$, the sequence $(Q(D)u_{n})_{n\geq 1}$ is relatively compact
in $L^{2}$. Then, $Q\prec _{c}P$.
\end{proposition}


\begin{proof}
The (weaker) domination $Q\prec P$ follows imeditely from Proposition \ref%
{prop.2}. It remains to show that $\widetilde{q}_{lj}/\widetilde{\Delta }%
\rightarrow 0$ at $\infty $, for any $l=1,...,L$, $j=1,...,M$.

Consider a sequence of functions $(\psi _{n})_{n\geq 1}$ in $\mathcal{E}%
^{\prime }(B)$ such that $\left\Vert \Delta (D)\psi _{n}\right\Vert
_{L^{2}}\leq 1$ and define $f_{n}:=(0,...,0,\psi _{n},0,...,0)$ (on the $j$%
-th position) and $u_{n}:=A(D)f_{n}\in \mathcal{E}^{\prime }(B)^{N}$, for
any $n\geq 1$. We have 
\begin{eqnarray*}
\left\Vert P(D)u_{n}\right\Vert _{L^{2}} &=&\left\Vert
P(D)A(D)f_{n}\right\Vert _{L^{2}}=\left\Vert P(D)P^{+}(D)\Delta
(D)f_{n}\right\Vert _{L^{2}} \\
&\lesssim &\left\Vert PP^{+}\Delta \widehat{f}_{n}\right\Vert
_{L^{2}}\lesssim \left\Vert \Delta \widehat{f}_{n}\right\Vert _{L^{2}}\sim
\left\Vert \Delta (D)f_{n}\right\Vert _{L^{2}} \\
&=&\left\Vert \Delta (D)\psi _{n}\right\Vert _{L^{2}}\leq 1,
\end{eqnarray*}%
where we have used the fact that the entries of $PP^{+}$are uniformly
bounded (see for instance the argument after (\ref{prop2-2})).

Now we have that $(Q(D)u_{n})_{n\geq 1}$ is relatively compact in $L^{2}$,
i.e., $(Q(D)A(D)f_{n})_{n\geq 1}$ is relatively compact in $L^{2}$. By the
definition of $f_{n}$ (recall the notation $QA=(q_{lj})_{l,j}$), this
implies that $(q_{lj}(D)\psi _{n})_{n\geq 1}$ is relatively compact in $%
L^{2} $, for any $l=1,...,L$, $j=1,...,M$.

We have obtained that for any sequence $(\psi _{n})_{n\geq 1}$ in $\mathcal{E%
}^{\prime }(B)$ such that $\left\Vert \Delta (D)\psi _{n}\right\Vert
_{L^{2}}\leq 1$, the sequence $(q_{lj}(D)\psi _{n})_{n\geq 1}$ is relatively
compact in $L^{2}$, for any $l,j\in \{1,...,N\}$. By Theorem 10.3.2 in \cite%
{Hor2} it is easy to observe that $B_{2,\widetilde{\Delta }}\cap \mathcal{E}%
^{\prime }(B)$ embedds compactly in $B_{2,\widetilde{q}_{lj}}$. Using now
Theorem \ref{th.cH} we get $\widetilde{q}_{lj}/\widetilde{\Delta }%
\rightarrow 0$ at $\infty $ and Proposition \ref{prop.4} is proved.
\end{proof}


\section{Application to lower semicontinuity}

The goal of this final section is to establish Theorem \ref{thm:main_lsc}. As mentioned in the introduction, having the neat boundary conditions enable
us to prove a statement for signed integrands which does not allow for
effects of concentration at the boundary.

Towards Theorem \ref{thm:main_lsc}, we can assume without loss of generality that $%
\Omega $ is a cube. We will carefully choose a cube that will enable us to
deal with concentration effects easily in the main proof.

Given a cube $\Omega \subset \mathbf{R}^{d}$ (with the axes parallel to the
coordinate axes) and a dyadic integer $m$, we denote by $\mathcal{F}%
_{m}(\Omega )$ the family of the $m^{d}$ pairwise disjoint and congruent
small cubes obtained by partitioning $\Omega $. We denote by $\mathbf{E}%
_{m}^{\Omega }$ the conditional expectation operator with respect to the $%
\sigma $-algebra of the cubes in $\mathcal{F}_{m}(\Omega )$. In order to
simplify the notation, when $\Omega $ is understood from the context, we
write $\mathbf{E}_{m}$ instead of $\mathbf{E}_{m}^{\Omega }$. Also, for a
given $\delta \in (0,1)$ we consider the set 
\begin{equation*}
\Omega _{\delta }^{m}:=\bigcup\limits_{\mathcal{Q}\in \mathcal{F}_{m}(\Omega
)}(1-\delta )\mathcal{Q}.
\end{equation*}

\begin{lemma}
\label{lem.Q}Fix some $\varepsilon >0$ and some function $v\in L_{c}^{2}(B)^N
$. Let $(v_{j})_{j\geq 1}$ be a bounded sequence in $L_{c}^{2}(B)^{N}$.

Then, there exists a dyadic integer $m$ and a cube $\Omega \supset B$ such
that 
\begin{equation*}
\left\Vert \mathbf{E}_{m}v-v\right\Vert _{L^{2}}<\varepsilon ,
\end{equation*}%
and, for any sufficiently small $\delta >0$, we have 
\begin{equation*}
\liminf_{j\rightarrow \infty }\left\Vert v_{j}\right\Vert _{L^{2}(\Omega
\backslash \Omega _{\delta }^{m})}<\varepsilon ,
\end{equation*}
\end{lemma}


\begin{proof}
Fix some $\varepsilon \in (0,1)$. Let $\Omega _{0}$ be a cube such that $%
B\subset \Omega _{\tau }:=\Omega _{0}+\tau $, for any $\tau \in B$ (in other
words, $B+B\subset \Omega _{0}$). Introduce the sequence of functions $%
(g_{m})_{m-dyadic}$, defined on $B$ by 
\begin{equation*}
g_{m}(\tau ):=\left\Vert \mathbf{E}_{m}^{\Omega _{\tau }}v-v\right\Vert
_{L^{2}},
\end{equation*}

We have the uniform bound 
\begin{equation*}
g_{m}(\tau )\leq \left\Vert \mathbf{E}_{m}^{\Omega _{\tau }}v\right\Vert
_{L^{2}}+\left\Vert v\right\Vert _{L^{2}}\leq 2\left\Vert v\right\Vert
_{L^{2}},
\end{equation*}%
for any $\tau \in B$ and any dyadic integer $m$. One can also see that each $%
g_{m}$is continuous on $B$. Indeed, for any function $f\in L^{1}(B)$ and any
cube $\mathcal{Q}$ we have, by the dominated convergence theorem,%
\begin{equation}
\int_{\mathcal{Q+\tau }_{n}}fdx\rightarrow \int_{\mathcal{Q+\tau }}fdx,
\label{Q-tau}
\end{equation}%
for any sequence $(\tau _{n})_{n\geq 1}$ in $B$ converging to some $\tau \in
B$. Now, by writing 
\begin{eqnarray*}
g_{m}^{2}(\tau ) &=&\left\Vert \mathbf{E}_{m}^{\Omega _{\tau }}v\right\Vert
_{L^{2}}^{2}-\left\langle \mathbf{E}_{m}^{\Omega _{\tau }}v,v\right\rangle
-\left\langle \mathbf{E}_{m}^{\Omega _{\tau }}v,v\right\rangle +\left\Vert
v\right\Vert _{L^{2}}^{2} \\
&=&-\sum_{\mathcal{Q}\in \mathcal{F}_{m}(\Omega _{0})}\frac{1}{|\mathcal{Q}|}%
\left\vert \int_{\mathcal{Q+\tau }}v(y)dy\right\vert ^{2}+\left\Vert
v\right\Vert _{L^{2}}^{2},
\end{eqnarray*}%
and using (\ref{Q-tau}) we get the continuity of $g_{m}$.

Also, we have $g_{m}(0)\rightarrow 0$, when $m\rightarrow \infty $
dyadically. Hence, for a fixed $m$ sufficiently large, we have $%
g_{m}(0)<\varepsilon /2$. By the continuity of $g_{m}$, we get $g_{m}(\tau
)<\varepsilon $, for any $\tau \in B(0,r)$, for some $r\in (0,1)$. In other
words, 
\begin{equation}
\left\Vert \mathbf{E}_{m}^{\Omega _{\tau }}v-v\right\Vert
_{L^{2}}<\varepsilon ,  \label{Exp}
\end{equation}%
for any $\tau \in B(0,r)$.

We define $\delta _{0}:=r\varepsilon ^{2}/(8C^{2}d)$ and the integer%
\begin{equation*}
M:=\left\lceil \frac{2dC^{2}}{\varepsilon ^{2}}\right\rceil ,
\end{equation*}%
where 
\begin{equation}
C:=\sup_{j\geq 1}\left\Vert v_{j}\right\Vert _{L^{2}}<\infty .  \label{vj-C}
\end{equation}

Consider the numbers $t_{1},...,t_{M}\in \lbrack 0,r/2]$ with $%
t_{l}:=rl/(2M) $, for any $l\in \{1,...,M\}$, and introduce the vectors $%
\tau _{1},...,\tau _{M}\in B(0,r)$, with $\tau _{l}:=t_{l}e$, for any $l\in
\{1,...,M\}$, where $e:=(1,...,1)\in \mathbf{R}^{d}$.

For a number $\delta \in (0,\delta _{0}]$ we consider the set 
\begin{equation*}
G_{\delta }:=\Omega _{0}\backslash (\Omega _{0})_{\delta }^{m},
\end{equation*}%
and observe that for any $\tau \in B$, 
\begin{equation}
G_{\delta }+\tau :=\Omega _{0}\backslash (\Omega _{0})_{\delta }^{m}+\tau
=\Omega _{\tau }\backslash (\Omega _{\tau })_{\delta }^{m}.  \label{G-tau}
\end{equation}

Since $\tau _{l}=l\tau _{1}$ and 
\begin{equation*}
|\tau _{1}|>\sqrt{d}r/(2M)>\sqrt{d}\delta ,
\end{equation*}%
if $\tau _{l_{1}},...,\tau _{l_{s}}$ are pairwise distinct, then the
intersection 
\begin{equation*}
(G_{\delta }+\tau _{l_{1}})\cap ...\cap (G_{\delta }+\tau _{l_{s}}),
\end{equation*}%
is embedded in $\delta -$neighborhoods of affine spaces of dimension $d-s$
between which the distance is at least $1-\delta $. (By convention, an
affine space of negative dimension is the empty set.) In particular, we get 
\begin{equation*}
\sum_{l=1}^{M}\mathbf{1}_{Sk_{\delta }+\tau _{l}}(x)\leq d,
\end{equation*}
for any $x\in\mathbf{R}^d$. Using this and (\ref{vj-C}) we can write

\begin{eqnarray*}
M\min_{l\in \{1,...,M\}}\left\Vert v_{j}\right\Vert _{L^{2}(G_{\delta }+\tau
_{l})}^{2} &\leq &\sum_{l=1}^{M}\left\Vert v_{j}\right\Vert
_{L^{2}(G_{\delta }+\tau _{l})}^{2}=\int_{\mathbf{R}^{d}}|v_{j}(x)|^{2}%
\left( \sum_{l=1}^{M}\mathbf{1}_{G_{\delta }+\tau _{l}}(x)\right) dx \\
&\leq &d\int_{\mathbf{R}^{d}}|v_{j}(x)|^{2}dx\leq dC^{2},
\end{eqnarray*}%
which gives us 
\begin{equation}
\min_{l\in \{1,...,M\}}\left\Vert v_{j}\right\Vert _{L^{2}(G_{\delta }+\tau
_{l})}^{2}\leq \frac{dC^{2}}{M}<\varepsilon ^{2}.  \label{vj-b}
\end{equation}

If 
\begin{equation*}
A_{l}:=\{j\in \mathbf{N}^{\ast }\text{ }|\text{ }\left\Vert v_{j}\right\Vert
_{L^{2}(G_{\delta }+\tau _{l})}<\varepsilon \},
\end{equation*}%
then (\ref{vj-b}) gives us $\bigcup_{l=1}^{M}A_{l}=\mathbf{N}^{\ast }$.
Hence, there exists one set $A_{l_{0}}$ that is infinite. It follows that 
\begin{equation*}
\liminf_{j\rightarrow \infty }\left\Vert v_{j}\right\Vert _{L^{2}(G_{\delta
}+\tau _{l_{0}})}<\varepsilon .
\end{equation*}
Due to this inequality, (\ref{G-tau}) and (\ref{Exp}) (recall that $\tau
_{l_{0}}\in B(0,r)$) one can choose $\Omega :=\Omega _{0}+\tau _{l_{0}}$.
\end{proof}

We can now proceed with the proof of the main lower semicontinuity result:

\begin{proof}[Proof of Theorem \ref{thm:main_lsc}.]
We first apply Lemma \ref{lem.Q} to the functions $v:=Su$ and $v_{j}:=Sw_{j}$%
, where $w_{j}:=u_{j}-u$ (note that since $v_{j}$ is weakly convergent, it
is also bounded in $L^{2}$). Fix some $\varepsilon >0$. There exists a
dyadic integer $m$ and a cube $\Omega \supset B$ such that 
\begin{equation}
\left\Vert \mathbf{E}_{m}Su-Su\right\Vert _{L^{2}}<\varepsilon ,
\label{LSC-1}
\end{equation}%
and, for any sufficiently small $\delta >0$, we have 
\begin{equation}
\liminf_{j\rightarrow \infty }\left\Vert Sw_{j}\right\Vert _{L^{2}(\Omega
\backslash \Omega _{\delta }^{m})}<\varepsilon .  \label{LSC-2}
\end{equation}

We can suppose that $(u_{j})_{j\geq 1}$ corresponds to a subsequence
realizing the $\inf $ in (\ref{LSC-2}).

Note that since $u$ and $u_{j}$ are supported on $B$, we have $%
F(Su_{j})-F(Su)=0$, outside of $B$. Thanks to this we can integrate on the
cube $\Omega$ instead of $B$: 
\begin{equation}
\int_{B}(F(Su_{j})-F(Su))dx=\int_{\Omega }(F(Su_{j})-F(Su))dx.
\label{B-Omega}
\end{equation}

Write now%
\begin{equation}
\int_{\Omega }(F(Su_{j})-F(Su))dx=T_{1}+T_{2}^{j}+T_{3}^{j},  \label{descF}
\end{equation}%
where 
\begin{equation*}
T_{1}:=\int_{\Omega }(F(\mathbf{E}_{m}Su)-F(Su))dx,
\end{equation*}%
\begin{equation*}
T_{2}^{j}:=\int_{\Omega }(F(Su_{j})-F(\mathbf{E}_{m}Su+Sw_{j}))dx,
\end{equation*}%
and 
\begin{equation*}
T_{3}^{j}:=\int_{\Omega }(F(\mathbf{E}_{m}Su+Sw_{j})-F(\mathbf{E}_{m}Su))dx.
\end{equation*}

\textit{We estimate }$T_{1}$\textit{. }By using Lemma 4.6 in\textit{\ }\cite%
{GR} and (\ref{LSC-1}) (and standard estimates) we can write%
\begin{eqnarray}
|T_{1}| &\leq &\int_{\Omega }|F(\mathbf{E}_{m}Su)-F(Su)|dx\lesssim
\int_{\Omega }(1+|\mathbf{E}_{m}Su|+|Su|)(|\mathbf{E}_{m}Su-Su|)dx  \notag \\
&\lesssim &\left( \int_{\Omega }(1+|\mathbf{E}_{m}Su|^{2}+|Su|^{2})dx\right)
^{1/2}\left\Vert \mathbf{E}_{m}Su-Su\right\Vert _{L^{2}}  \notag \\
&\lesssim &(|\Omega |^{1/2}+\left\Vert \mathbf{E}_{m}Su\right\Vert
_{L^{2}}+\left\Vert Su\right\Vert _{L^{2}})\left\Vert \mathbf{E}%
_{m}Su-Su\right\Vert _{L^{2}}  \label{T1-3} \\
&\lesssim &(1+\left\Vert Su\right\Vert _{L^{2}})\left\Vert \mathbf{E}%
_{m}Su-Su\right\Vert _{L^{2}}  \notag \\
&\lesssim &_{u}\left\Vert \mathbf{E}_{m}Su-Su\right\Vert
_{L^{2}}<\varepsilon .  \label{T1-5}
\end{eqnarray}

\textit{We estimate }$T_{2}^{j}$\textit{. }Note that $Su_{j}-\mathbf{E}%
_{m}Su-Sw_{j}=-(\mathbf{E}_{m}Su-Su)$. As in (\ref{T1-3}) we write 
\begin{eqnarray}
|T_{2}^{j}| &\lesssim &(|\Omega |^{1/2}+\left\Vert Su_{j}\right\Vert
_{L^{2}}+\left\Vert \mathbf{E}_{m}Su+Sw_{j}\right\Vert _{L^{2}})\left\Vert
Su_{j}-\mathbf{E}_{m}Su-Sw_{j}\right\Vert _{L^{2}}  \notag \\
&=&(|\Omega |^{1/2}+\left\Vert Su_{j}\right\Vert _{L^{2}}+\left\Vert \mathbf{%
E}_{m}Su+Sw_{j}\right\Vert _{L^{2}})\left\Vert \mathbf{E}_{m}Su-Su\right%
\Vert _{L^{2}}  \notag \\
&\lesssim &(1+\left\Vert Su_{j}\right\Vert _{L^{2}}+\left\Vert Su\right\Vert
_{L^{2}}+\left\Vert Sw_{j}\right\Vert _{L^{2}})\left\Vert \mathbf{E}%
_{m}Su-Su\right\Vert _{L^{2}}  \notag \\
&\lesssim &\varepsilon ,  \label{T2}
\end{eqnarray}%
where we have used the fact that, since $Su_{j}$ converges weakly to $Su$
and $Sw_{j}$ converges weakly to $0$, the quantities $\left\Vert
Su_{j}\right\Vert _{L^{2}}$ and $\left\Vert Sw_{j}\right\Vert _{L^{2}}$ are
uniformly bounded. (The implicit constant in (\ref{T2}) depends on $u$ and
on the sequence $(u_{j})_{j\geq 1}$.)

\bigskip

\textit{We estimate }$T_{3}^{j}$\textit{. }We consider a function $\eta \in
C_{c}^{\infty }(\Omega _{\delta /2}^{m})$ such that $\eta \equiv 1$ on $%
\Omega _{\delta }^{m}$, and denote by $w_{j,\delta }$ the functions $%
w_{j,\delta }:=\eta w_{j}$, for any $j\geq 1$. We decompose $T_{3}$ as $%
T_{3}^{j}=T_{31}^{j}+T_{32}^{j}$, where

\begin{equation*}
T_{31}^{j}:=\int_{\Omega }(F(\mathbf{E}_{m}Su+Sw_{j,\delta })-F(\mathbf{E}%
_{m}Su))dx,
\end{equation*}%
and

\begin{equation*}
T_{32}^{j}:=\int_{\Omega }(F(\mathbf{E}_{m}Su+Sw_{j})-F(\mathbf{E}%
_{m}Su+Sw_{j,\delta }))dx.
\end{equation*}

By the $S-$quasiconvexity of $F$ (see (\ref{qc})) we have 
\begin{equation}
T_{31}^{j}\geq 0.  \label{T31}
\end{equation}

In the case of $T_{32}^{j}$ we see that $F(\mathbf{E}_{m}Su+Sw_{j})-F(%
\mathbf{E}_{m}Su+Sw_{j,\delta })=0$ in $\Omega _{\delta }^{m}$ and we write%
\begin{eqnarray}
|T_{32}^{j}| &=&|\int_{\Omega \backslash \Omega _{\delta }^{m}}(F(\mathbf{E}%
_{m}Su+Sw_{j})-F(\mathbf{E}_{m}Su+Sw_{j,\delta }))dx|  \notag \\
&\leq &\int_{\Omega \backslash \Omega _{\delta }^{m}}|F(\mathbf{E}%
_{m}Su+Sw_{j})|dx+\int_{\Omega \backslash \Omega _{\delta }^{m}}|F(\mathbf{E}%
_{m}Su+Sw_{j,\delta })|dx  \notag \\
&=&:a_{1}^{j}+a_{2}^{j}.  \label{T32-1}
\end{eqnarray}

Since we have $|F(x)|\lesssim \left\langle x\right\rangle ^{2}$, on $\mathbf{%
R}^{d}$, we get%
\begin{equation*}
a_{1}^{j}\lesssim \int_{\Omega \backslash \Omega _{\delta }^{m}}(1+|\mathbf{E%
}_{m}Su|^{2}+|Sw_{j}|^{2})dx\leq |\Omega \backslash \Omega _{\delta
}^{m}|+\left\Vert \mathbf{E}_{m}Su\right\Vert _{L^{2}(\Omega \backslash
\Omega _{\delta }^{m})}^{2}+\left\Vert Sw_{j}\right\Vert _{L^{2}(\Omega
\backslash \Omega _{\delta }^{m})}^{2}.
\end{equation*}

By (\ref{LSC-2}) we have $\left\Vert Sw_{j}\right\Vert _{L^{2}(\Omega
\backslash \Omega _{\delta }^{m})}^{2}<\varepsilon ^{2}<\varepsilon $. If $%
\delta $ is sufficiently small, we have $|\Omega \backslash \Omega _{\delta
}^{m}|<\varepsilon $ and also, by the dominated convergence theorem, $%
\left\Vert Sw_{j}\right\Vert _{L^{2}(\Omega \backslash \Omega _{\delta
}^{m})}^{2}<\varepsilon $. Hence, 
\begin{equation}
|\Omega \backslash \Omega _{\delta }^{m}|+\left\Vert \mathbf{E}%
_{m}Su\right\Vert _{L^{2}(\Omega \backslash \Omega _{\delta
}^{m})}^{2}+\left\Vert Sw_{j}\right\Vert _{L^{2}(\Omega \backslash \Omega
_{\delta }^{m})}^{2}\lesssim \varepsilon ,  \label{a1'}
\end{equation}%
and 
\begin{equation}
a_{1}^{j}\lesssim \varepsilon .  \label{a1}
\end{equation}

We use again the inequality $|F(x)|\lesssim \left\langle x\right\rangle ^{2}$%
, on $\mathbf{R}^{d}$, in order to estimate $a_{2}^{j}$:%
\begin{eqnarray}
a_{2}^{j} &\lesssim &\int_{\Omega \backslash \Omega _{\delta }^{m}}(1+|%
\mathbf{E}_{m}Su|^{2}+|Sw_{j,\delta }|^{2})dx  \notag \\
&\lesssim &\int_{\Omega \backslash \Omega _{\delta }^{m}}(1+|\mathbf{E}%
_{m}Su|^{2}+|\eta Sw_{j}|^{2}+|[S,\eta ]w_{j}|^{2})dx  \notag \\
&\lesssim &|\Omega \backslash \Omega _{\delta }^{m}|+\left\Vert \mathbf{E}%
_{m}Su\right\Vert _{L^{2}(\Omega \backslash \Omega _{\delta
}^{m})}^{2}+\left\Vert Sw_{j}\right\Vert _{L^{2}(\Omega \backslash \Omega
_{\delta }^{m})}^{2}+\left\Vert [S,\eta ]w_{j}\right\Vert _{L^{2}(\Omega
\backslash \Omega _{\delta }^{m})}^{2},  \label{a2-u}
\end{eqnarray}%
where 
\begin{equation*}
\lbrack S,\eta ]=\eta S-S\eta .
\end{equation*}

Note that 
\begin{equation*}
\lbrack S,\eta ]=[P(D),\eta ]=\sum_{\alpha \neq 0}\frac{D^{\alpha }\eta }{%
\alpha !}P^{(\alpha )}(D).
\end{equation*}%
Since $Sw_{j}=P(D)w_{j}$ converges weakly to $0$ in $L^{2}$, by the
assumptions on the operator $S$ ($P^{(\alpha )}\prec _{c}P$ for any $\alpha
\neq 0$) and Corollary \ref{cor.conv0} we get that, for any $\alpha \neq 0$,
\ $P^{(\alpha )}(D)w_{j}\rightarrow 0$ strongly in $L^{2}$, on a
subsequence.\ From this we obtain that\ $[S,\eta ]w_{j}\rightarrow 0$
strongly in $L^{2}$, on a subsequence. Hence, 
\begin{equation*}
\left\Vert \lbrack S,\eta ]w_{j}\right\Vert _{L^{2}(\Omega \backslash \Omega
_{\delta }^{m})}^{2}\lesssim \varepsilon ,
\end{equation*}%
for large $j$ on a subsequence. By this, (\ref{a1'}) and (\ref{a2-u}) give
us 
\begin{equation*}
a_{2}^{j}\lesssim \varepsilon .
\end{equation*}

By this, (\ref{a1'}) and (\ref{T32-1}) we have 
\begin{equation*}
\liminf_{j\rightarrow \infty }|T_{32}^{j}|\lesssim \varepsilon ,
\end{equation*}%
which together with (\ref{T1-5}), (\ref{T2}), (\ref{T31}) and (\ref{descF})
(see also (\ref{B-Omega})) gives us 
\begin{equation*}
\liminf_{j\rightarrow \infty }\int_{B}(F(Su_{j}(x))-F(Su(x)))dx\gtrsim
-\varepsilon .
\end{equation*}

Since $\varepsilon $ is arbitrary small, we get the conclusion. 
\end{proof}




\end{document}